\documentclass[10pt,reqno]{amsart}

\usepackage{amsmath,amscd,amssymb,amsthm}
\usepackage{tikz-cd}
\usepackage{comment}

\makeatletter
\def\th@plain{%
  \thm@notefont{}
  \itshape 
}
\def\th@definition{%
  \thm@notefont{}
  \normalfont 
}
\makeatother

\usepackage[hyperindex]{hyperref}

\usepackage{enumerate}
\usepackage{todonotes}
\usepackage{color}

\newtheorem{proposition}{Proposition}[section]
\newtheorem{lemma}[proposition]{Lemma}
\newtheorem{theorem}[proposition]{Theorem}
\newtheorem{corollary}[proposition]{Corollary}

\theoremstyle{definition}
\newtheorem{remark}[proposition]{Remark}
\newtheorem{definition}[proposition]{Definition}

\numberwithin{equation}{section} 

\newcommand{\R}{\mathbb{R}}
\newcommand{\C}{\mathbb{C}}

\newcommand{\Z}{\mathbb{Z}}
\newcommand{\Q}{\mathbb{Q}}

\newcommand{\pr}{\mathbb{P}}

\newcommand{\scX}{\mathcal{X}}
\newcommand{\X}{\mathcal{X}}

\newcommand{\mft}{\mathfrak{t}}

\DeclareMathOperator{\Ric}{Ric}

\DeclareMathOperator{\DF}{DF}

\title[The K\"ahler-Ricci flow and optimal degenerations]{The K\"ahler-Ricci flow and optimal degenerations}

\author[Ruadha\'i Dervan]{Ruadha\'i Dervan}
\address{Ruadha\'i Dervan, Department of Pure Mathematics and Mathematical Statistics, University of Cambridge.}
\email{R.Dervan@dpmms.cam.ac.uk}

\author[G\'abor Sz\'ekelyhidi]{G\'abor Sz\'ekelyhidi}
\address{G\'abor Sz\'ekelyhidi, Department of Mathematics, University of Notre Dame.}

\email{gszekely@nd.edu}

\begin{document}

\begin{abstract} 

We prove that on Fano manifolds, the K\"ahler-Ricci flow produces a ``most destabilising'' degeneration, with respect to a new stability notion related to the $H$-functional. This answers questions of Chen-Sun-Wang and He.

We give two applications of this result. Firstly, we give a purely algebro-geometric formula for the supremum of Perelman's $\mu$-functional on Fano manifolds, resolving a conjecture of Tian-Zhang-Zhang-Zhu as a special case. Secondly, we use this to prove that if a Fano manifold admits a K\"ahler-Ricci soliton, then the K\"ahler-Ricci flow converges to it modulo the action of automorphisms, with any initial metric. This extends work of Tian-Zhu and Tian-Zhang-Zhang-Zhu, where either the manifold was assumed to admit a K\"ahler-Einstein metric, or the initial metric of the flow was assumed to be invariant under a maximal compact group of automorphism. 

\end{abstract}

\maketitle

\section{introduction}

A basic question in K\"ahler geometry is which Fano manifolds admit K\"ahler-Einstein metrics. The Yau-Tian-Donaldson conjecture~\cite{Yau93, don-toric, GT}, resolved by Chen-Donaldson-Sun~\cite{CDS}, relates this to K-stability of the manifold. 

\begin{theorem}\cite{CDS} A Fano manifold admits a K\"ahler-Einstein metric if and only if it is K-stable. \end{theorem}

In this paper our main focus is the situation when a Fano manifold $X$ does not admit a K\"ahler-Einstein metric. In analogy with the Harder-Narasimhan filtration of unstable vector bundles, and more generally with optimal destabilising one-parameter subgroups in geometric invariant theory~\cite{GK,BT05}, one expects that in this case $X$ has an optimal destabilising degeneration. 

One precise conjecture in this direction is due to Donaldson~\cite{SD}, predicting that the infimum of the Calabi functional on $X$ is given by the supremum of the Donaldson-Futaki invariants $\mathrm{DF}(\mathcal{X})$ over all test-configurations $\mathcal{X}$ for $X$. While this conjecture remains open in general, we show that an analogous result holds in the Fano case, if we replace the Calabi functional by the H-functional, and the Donaldson-Futaki invariant with the H-invariant, which we define by analogy with an invariant introduced by Tian-Zhang-Zhang-Zhu~\cite{TZZZ} for holomorphic vector fields.

\begin{theorem}\label{introtheorem} Let $X$ be a Fano manifold. We have 
\begin{equation}\label{eq:infH}
 \inf_{\omega\in c_1(X)} H(\omega) =\sup_{\X} H(\X),
\end{equation}
where $H(\omega)$ is the H-functional, and the supremum is taken over the H-invariants of all test-configurations $\X$ for $X$. 
\end{theorem}

Our proof builds closely on the previous works~\cite{CSW,CW,WH}. The H-functional is defined by
\[ H(\omega) = \int_X h e^h\, \omega^n, \]
where $h$ is the Ricci potential of $\omega$ normalised so that $e^h$ has average 1. As far as we can tell this functional first appears in the literature as the difference of the Mabuchi and Ding functionals in Ding-Tian~\cite[p. 69]{DT92b}, where it is called $E$, and it is observed that $H(\omega)\geq 0$ with equality only if $\omega$ is K\"ahler-Einstein. The functional was shown to be monotonic along the K\"ahler-Ricci flow in Pali~\cite{Pali08} and Phong-Song-Sturm-Weinkove~\cite{PSSW3}. He~\cite{WH} studied the functional on the space of K\"ahler metrics in more detail, viewing it as analogous to the Calabi functional and giving lower bounds for it in terms of Tian-Zhang-Zhang-Zhu's invariant~\cite{TZZZ} for vector fields. A ``moment map'' interpretation of the H-functional has been described by Donaldson~\cite{SD2}. For the precise definitions of a test-configuration and the H-invariant, see Definitions~\ref{defn:special} and \ref{defn:Hinv} below. 

In view of Theorem~\ref{introtheorem} it is natural to say that $X$ is H-stable, if $H(\X) < 0$ for all non-trivial special degenerations (which are simply test-configurations with $\Q$-Fano central fibre). It follows that then the infimum of $H(\omega)$ is zero, or equivalently that $X$ is ``almost K\"ahler-Einstein'' in the sense of \cite{SB} and in particular $X$ is K-semistable. Conversely in Lemma~\ref{lem:jensen} we show that K-semistability implies H-stability, hence H-stability is not enough to detect the existence of a K\"ahler-Einstein metric in general. 

Given Theorem~\ref{introtheorem} it is natural to ask whether the supremum on the right hand side of \eqref{eq:infH} is achieved by a test-configuration. Our next result shows that this is the case, as long as we allow a slight generalisation of the notion of test-configurations to what we call $\R$-degenerations (see Definition~\ref{defn:Rdegen}). In fact such an optimal $\R$-degeneration is given by the filtration constructed by Chen-Sun-Wang \cite{CSW} using the K\"ahler-Ricci flow. 

\begin{corollary} On a Fano manifold that does not admit a K\"ahler-Einstein metric, the K\"ahler-Ricci flow produces an optimal $\R$-degeneration, with maximal $H$-invariant.  
\end{corollary}

As we mentioned above, such an optimal degeneration should be thought of as analogous to the Harder-Narasimhan filtration for an unstable vector bundle. The corollary answers questions of Chen-Sun-Wang \cite[Question 3.8]{CSW} and He \cite[Question 3]{WH2}. It is natural to conjecture that optimal degenerations in this sense are unique. By the above result, this is related to a conjecture of Chen-Sun-Wang regarding uniqueness of the degenerations induced by the K\"ahler-Ricci flow  \cite[Conjecture 3.7]{CSW}.

Our approach enables us to prove new results about the K\"ahler-Ricci flow on Fano manifolds. First we give an algebro-geometric interpretation of the supremum of Perelman's $\mu$-functional~\cite{Per02} on Fano manifolds that do not admit K\"ahler-Einstein metrics. This builds again on work of He \cite{WH}, and as a special case answers Conjecture 3.4 in Tian-Zhang-Zhang-Zhu~\cite{TZZZ}. 

\begin{theorem}Let $\mu(\omega)$ denote Perelman's $\mu$-functional. We have $$\sup_{\omega\in c_1(X)}\mu(\omega) = nV - \sup_{\X} H(\X).$$ 
In addition the $\mu$-functional tends to its supremum along the K\"ahler-Ricci flow starting from any initial metric.
\end{theorem}

\noindent As shown by Berman~\cite{RB}, the Donaldson-Futaki invariant of special degenerations can also yield upper bounds for the $\mu$-functional, however it is not known whether the supremum can be characterised in that way. 

Finally, using a result of Tian-Zhang-Zhang-Zhu \cite{TZZZ}, we obtain a general convergence result for the K\"ahler-Ricci flow, assuming the existence of a K\"ahler-Ricci soliton. 

\begin{corollary} Suppose $X$ is a Fano manifold admitting a K\"ahler-Ricci soliton $\omega_{KRS}$, and let $\omega\in c_1(X)$ be an arbitrary K\"ahler metric. The K\"ahler-Ricci flow starting from $\omega$ converges to $\omega_{KRS}$, up to the action of the automorphism group of $X$.
\end{corollary} 

When $\omega_{KRS}$ is actually K\"ahler-Einstein, this was proven by Tian-Zhu~\cite{TZ1,TZ2} using Perelman's estimates~\cite{Per02, ST03} (see also \cite{CS2}). When $\omega_{KRS}$ is a general K\"ahler-Ricci soliton, then Tian-Zhang-Zhang-Zhu~\cite{TZZZ} proved the convergence result under the assumption that the initial metric $\omega$ is invariant under a maximal compact group of automorphisms of $X$. 

We emphasise that the techniques used in the present note are not really new. Instead the novelty in our work is the observation that by using a slightly different notion of stability, the older techniques yield new stronger results.

\

\noindent {\bf Acknowledgements:} The first author would like to thank Stuart Hall and Song Sun for helpful comments. We are also grateful to Thibaut Delcroix and Shijin Zhang for comments on a previous version of the paper. We also thank the referee for their comments. Some of this work was carried out while the first author visited Notre Dame; he thanks the department for their hospitality. The second author is grateful for support from NSF grant DMS-1350696. 

\

\noindent {\bf Notation and conventions:} We normalise $dd^c$ so that $dd^c = \frac{i}{2\pi}\partial\bar\partial$. A $\Q$-Fano variety is a normal variety $X$ such that $-K_X$ is an ample $\Q$-line bundle and such that $X$ has log terminal singularities. The volume of $X$ is denoted by $V=(-K_X)^n = \int_X c_1(X)^n$. When $X$ is smooth, given a metric $\omega\in c_1(X)$, we define its Ricci potential $h$ to be such that $\Ric\omega - \omega = dd^c h$ and $\int_X e^h\omega^n = V.$ When $X$ is a $\Q$-Fano variety, given a positive metric $p$ on $-K_X$ with curvature $\omega$, we set the Ricci  potential to be $h = \log\left(\frac{\omega^n}{p}\right)+C$, with the constant $C$ chosen such that $\int_X e^h\omega^n = V.$ By a smooth Kahler metric on X we mean a $(1,1)$-form on the regular part of $X$ obtained locally by restricting a smooth K\"ahler metric under a local embedding into $\C^N$, see for example \cite{JPD}.

\section{H-stability}

\subsection{Analytic aspects}  

Let $X$ be a $\Q$-Fano variety. We first focus on the analytic aspects of H-stability, postponing the algebraic description to Section \ref{sec:alg}. Here we will restrict ourselves to considering special degenerations in the sense of Tian~\cite{GT}. In Section~\ref{sec:alg} we will define $\R$-degenerations. 

\begin{definition}\label{defn:special}
\cite{GT} A \emph{special degeneration} of $X$ is a normal $\Q$-Fano family $\pi:\X\to\C$, together with a holomorphic vector field $v$ on $\X$, a real multiple of which generates a $\C^*$-action on $\scX$ covering the natural action on $\C$. In addition the fibre $\scX_t$ over $t$ is required to be isomorphic to $X$ for one, and hence all, $t \in \C^*$. 
We call $\X$ a product special degeneration if $\X \cong X\times \C$, while $\X$ is trivial if in addition the vector field is trivial on the $X$ factor. \end{definition}

We now define the H-invariant to be a certain integral over $\X_0$, which was first considered by Tian-Zhang-Zhang-Zhu~\cite{TZZZ} in the case of product special degenerations with $X$ smooth. We remark that this H-invariant is not linear in the vector field $v$, and so it is important to allow ``scalings'' of $\C^*$-actions in the above definition. 

\begin{definition} \label{defn:Hinv}
Choose a smooth K\"ahler metric $\omega_0\in c_1(\X_0)$, and let $h_0$ be the corresponding Ricci potential. Let $\theta_0$ be a Hamiltonian for the induced holomorphic vector field $v$ on $\X_0$. Then we define the \emph{H-invariant} $H(\X)$ to be  $$H(\X) = \int_{\X_0} \theta_0 e^{h_0}\omega_0^n - V \log\left(\frac{1}{V}\int_{\X_0}e^{\theta_0} \omega_0^n\right) ,$$
where the integrals can be performed on the regular part of $\X_0$. \end{definition} 

The factors of $V$ are included to ensure the integral equals zero if $\theta_0=0$. It is straight forward to check that $H(\X)$ is independent of choice of Hamiltonian, i.e. it does not change if we replace $\theta$ by $\theta+C$ for $C\in\R$.
We will later give an equivalent, algebro-geometric, definition of the H-invariant which will show that it is also independent of choice of $\omega_0\in c_1(\X_0)$ (which follows from \cite{TZZZ} with $\X_0$ smooth). 

The definition of H-stability then simply requires control of the sign of the H-invariant of each special degeneration.

\begin{definition} We say that $X$ is \emph{H-stable} if $H(\X)<0$ for all non-trivial special degenerations of $X$. \end{definition}

By contrast recall the analytic definition of the Donaldson-Futaki invariant, which is the relevant weight for K-stability.

\begin{definition}We define the \emph{Donaldson-Futaki invariant} of $\X$ to be $$\DF(\X) =  \int_{\X_0} \theta_0 e^{h_0}\omega_0^n - \int_{\X_0} \theta_0 \omega_0^n .$$ We say that $X$ is \emph{K-semistable} if $\DF(\X)\leq 0$ for all special degenerations $\X$, and $X$ is \emph{K-stable} if in addition $\DF(\X)=0$ only if $\X_0 \cong X$.\end{definition}

We remark that, in the literature, the opposite sign convention for the Donaldson-Futaki invariant is often used, but the present sign convention seems to be more natural in view of \eqref{eq:infH}. The conventions also sometimes differ in the sign given to the Hamiltonian function corresponding to a $\C^*$-action. This makes little difference in the Donaldson-Futaki invariant, since it is linear, but the H-invariant is more sensitive to changes in such conventions. 

The two invariants are related as follows:

\begin{lemma}\label{lem:jensen} We have $\DF(\X) \geq H(\X)$, with equality if and only if $\X$ is trivial. \end{lemma}

\begin{proof} This follows by Jensen's inequality. Indeed, $V^{-1}\omega_0^n$ is a probability measure so by concavity of the logarithm $$\log \int_{\X_0} e^{\theta_0}(V^{-1}\omega_0^n) \geq \int_{\X_0} \theta_0 (V^{-1}\omega_0^n),$$ with equality if and only if $\theta_0$ is constant, i.e. $\X$ is trivial.  \end{proof}

\begin{corollary} If a Fano manifold  is K-semistable, then it is H-stable. In particular, if a Fano manifold admits a K\"ahler-Einstein metric, then it is H-stable, but the converse is not necessarily true. \end{corollary}

\begin{proof} The first part follows directly from Lemma \ref{lem:jensen}. The second follows from the fact that the existence of a K\"ahler-Einstein metric implies K-semistability, but there are K-semistable manifolds that are not K-stable, and so do not admit a K\"ahler-Einstein metric (see Ding-Tian \cite{DT92} and Tian~\cite{GT}). \end{proof}

\subsection{Algebraic aspects}\label{sec:alg}

We first define the algebro-geometric degenerations that we will consider. This will be a generalisation of the notion of a test-configuration~\cite{don-toric}, using the language of filtrations~\cite{DWN, BHJ}. Let $X$ be an arbitrary projective variety, and for any ample line bundle $L \to X$ let us define the graded coordinate ring
\[ R(X, L) = \bigoplus_{m\geq 0} H^0(X, mL). \]
For simplicity we write $R_m = H^0(X, mL)$. By a filtration of $R(X,L)$ we mean an $\R$-indexed filtration $\{F^\lambda R_m\}_{\lambda\in \R}$ for each $m$, satisfying
\begin{itemize}
\item[(i)] $F$ is decreasing: $F^\lambda R_m\subset F^{\lambda'}R_m$ whenever $\lambda \geq \lambda'$, 
\item[(ii)] $F$ is left-continuous: $F^\lambda R_m = \bigcap_{\lambda' < \lambda} F^{\lambda'}R_m$,
\item[(iii)] For each $m$,  $F^\lambda R_m=0$ for $\lambda \gg 0$ and $F^\lambda R_m= R_m$ for $\lambda \ll 0$, 
\item[(iv)] $F$ is multiplicative: 
\[ F^\lambda R_m\cdot F^{\lambda'} R_{m'} \subset F^{\lambda+\lambda'}R_{m+m'}. \]
\end{itemize}
The associated graded ring of such a filtration is defined to be 
\[ \mathrm{gr} F^\lambda R(X,L) = \bigoplus_{m\geq 0} \bigoplus_{i}  F^{\lambda_{m,i}}R_m\big{/} F^{\lambda_{m,i+1}} R_m, \]
where the $\lambda_{m,i}$ are the values of $\lambda$ where the filtration of $R_m$ is discontinuous. 

\begin{definition} \label{defn:Rdegen}
  An $\R$-degeneration for $(X, L)$ is a filtration of $R(X, rL)$ for some integer $r > 0$, whose associated graded ring is finitely generated.    
\end{definition}

Let us recall here that as shown by Witt Nystr\"om~\cite{DWN} (see also \cite[Proposition 2.15]{BHJ}), test-configurations correspond to $\Z$-filtrations of $R(X, rL)$ whose Rees algebra is finitely generated. These are filtrations satisfying $F^\lambda R_m = F^{\lceil\lambda\rceil} R_m$ and that
\[ \bigoplus_{m\geq 0}\left( \bigoplus_{\lambda\in \Z} t^{-\lambda} F^\lambda R_m\right) \]
is a finitely generated $\C[t]$-algebra. 

The notion of an $\R$-degeneration arises naturally from the work of Chen-Sun-Wang~\cite{CSW} on the K\"ahler-Ricci flow, and it corresponds to considering degenerations of a projective variety under certain real one-parameter subgroups of $GL_N(\C)$. Moreover, as Chen-Sun-Wang do in \cite[Section 3.1]{CSW}, we can approximate any $\R$-degeneration by a sequence of test-configurations. Indeed, given an $\R$-degeneration for $(X, L)$, let us denote by $\overline{R}$ the associated graded ring of the filtration, and by $X_0 = \mathrm{Proj}\, \overline{R}$ the corresponding projective variety. The filtration gives rise to a real one-parameter family of automorphisms of $\overline{R}$, by scaling the elements in 
\[  F^{\lambda_{m,i}}R_m\big{/} F^{\lambda_{m,i+1}} R_m  \]
by $t^{\lambda_m, i}$. Without loss of generality we can suppose that $\overline{R}$ is generated by the elements $\overline{R}_1$ of degree one (i.e. those coming from sections of $L$ on $X$), and we set $N +1 = \dim H^0(X,L)$. In this case we have an embedding $X_0 \subset \pr^N$, and a real one-parameter group of projective automorphisms of $X_0$ given by $e^{t\Lambda}$, where $\Lambda$ is a diagonal matrix with eigenvalues $\lambda_{1, i}$. In addition, as discussed in \cite{CSW}, we have an embedding $X\subset \mathbf \pr^N$, such that $\lim_{t\to\infty} e^{t\Lambda}\cdot X = X_0$ in the Hilbert scheme. Note that the matrices $e^{\sqrt{-1} t \Lambda}$ generate a compact torus $T$ in $U(N)$, which must preserve $X_0$, and therefore also acts on the coordinate ring $\overline{R}$.  Perturbing the eigenvalues of $\Lambda$ slightly to rational numbers, we can obtain a $\C^*$-subgroup $\rho$ of the complexified torus $T^\C$ acting on $\pr^N$, for which $\lim_{t\to 0} \rho(t)\cdot X = X_0$, i.e. we have a test-configuration for $X$ with the central fiber $X_0$, and the induced action on $X_0$ is a perturbation of $e^{t\Lambda}$. 

In terms of the coordinate ring $\overline{R}$ the action of $T^\C$ together with the $m$-grading defines the action of $\C^*\times T^\C$, and allows us to decompose
\[ \overline{R} = \bigoplus_{m \geq 0} \bigoplus_{\alpha\in \mathfrak{t}^*} \overline{R}_{m,\alpha} \]
into weight spaces. The action of $e^{t\Lambda}$ corresponds to a choice of (possibly irrational) $\xi\in \mathfrak{t}$, so that the weight of its action on $\overline{R}_{m,\alpha}$ is $\langle \alpha, \xi\rangle$, and the approximating test-configurations obtained above give rise to a sequence of rational $\xi_k\in \mathfrak{t}$ such that $\xi_k \to \xi$. We will define the H-invariant of an $\R$-degeneration in such a way that it is continuous under this approximation procedure. 

\begin{definition} Let $\eta \in \mft$. For $t\in\C$, we define the \emph{weight character} by $$C(\eta,t) = \sum_{m \geq 0, \alpha\in\mft^*} e^{-tm} \alpha(\eta)\dim \overline{R}_{m,\alpha} .$$ \end{definition}

By \cite[Theorem 4]{CS}, the weight character is a meromorphic function in a neighbourhood of $0\in\C$, with the following Laurent series expansion: $$ C(\eta,t) = \frac{b_0(n+1)!}{t^{n+2}} + \frac{b_1(n+2)!}{t^{n+1}}  + O(t^{-n}).$$ Moreover, $b_0,b_1$ are smooth functions of $\eta$. For the H-invariant, the key term will be the $b_1$ term of this expansion.

\begin{remark} When the $\eta$ is integral, the above construction reduces to a more well known equivariant Riemann-Roch construction, namely the constants $b_0, b_1$ are given by the asymptotics of the total weight
\[ \sum_{\alpha\in \mathfrak{t}^*} \alpha(\eta) \cdot \dim \overline{R}_{m,\alpha} = b_0 m^{n+1} + b_1 m^n + O(m^{n-1}). \]
\end{remark}

To define the H-invariant, we will also need the term 
\begin{equation}
\label{eq:c0}
 \frac{c_0}{n!} = \lim_{m\to\infty}\sum_{\alpha\in \mathfrak{t}^*} m^{-n} e^{- m^{-1} \alpha(\eta)} \dim \overline{R}_{m,\alpha}, 
\end{equation}
which is also a function of $\eta$. 
\begin{lemma}
  The limit $c_0$ is well defined, and continuous in $\eta$. 
\end{lemma}
\begin{proof}
The existence of this limit can be seen as a consequence of the existence of the Duistermaat-Heckman measure 
\[ \mathrm{DH} = \lim_{m\to\infty} \frac{1}{N_m} \sum_{\alpha\in \mathfrak{t}^*} \dim \overline{R}_{m,\alpha} \, \delta_{m^{-1}\alpha(\eta)}, \]
where $N_m = \dim \overline{R}_m$ and $\delta$ denotes the Dirac measure (see for instance Boucksom-Hisamoto-Jonsson~\cite[Section 5]{BHJ}). In our situation this measure has compact support, and 
\begin{equation}\label{eq:c0int}
 \frac{c_0}{V} = \int_{\R} e^{-\lambda}\, \mathrm{DH}(\lambda). 
\end{equation}

In order to show that $c_0$ in \eqref{eq:c0} is continuous in $\eta$, we argue by approximation. We can assume, as above, that $\overline{R}$ is generated by the degree one elements $\overline{R}_1$. This implies that if $\alpha\in\mathfrak{t}^*$ is such that $\overline{R}_{m,\alpha}$ is nonzero, then $\alpha$ is a sum of $m$ weights appearing in the action on $\overline{R}_1$, and so in particular $|\alpha| < Cm$ for a uniform $C$. 

Suppose now that $\eta, \eta' \in \mathfrak{t}$, and $\overline{R}_{m,\alpha}$ is nontrivial. Then using the mean value theorem we obtain
\[ \left|e^{-m^{-1}\alpha(\eta)} - e^{-m^{-1} \alpha(\eta')}\right| \leq C m^{-1} | \alpha(\eta) - \alpha(\eta')| \leq C |\eta - \eta'|, \]
for a uniform $C$. Since $\dim \overline{R}_{m,\alpha} \leq \dim \overline{R}_m \leq d m^n$ for some $d > 0$,  it follows that for each $m$ we have
\[ \left| \sum_{\alpha\in \mathfrak{t}^*} m^{-n} \left(e^{- m^{-1} \alpha(\eta)}- e^{- m^{-1} \alpha(\eta')} \right) \dim \overline{R}_{m,\alpha} \right| \leq Cd |\eta - \eta'|. \]
It now follows that if $\eta_k \to \eta$, and for each $\eta_k$ the limit \eqref{eq:c0} defining $c_0(\eta_k)$ exists, then the limit defining $c_0(\eta)$ also exists and $c_0(\eta_k) \to c_0(\eta)$. 
\end{proof}

We now return to the case that $X$ is a $\Q$-Fano variety.

\begin{definition}\label{defn:algH}
 Let $\X$ be an $\R$-degeneration for $(X, -K_X)$. 
We define the \emph{H-invariant} of $\X$ to be $$H(\X) = -V\log\left(\frac{c_0}{V}\right) - 2((n-1)!)b_1,$$
where $V = (-K_X)^n$ is the volume as before. 
\end{definition}

This definition agrees with the analytic one for special degenerations: 

\begin{proposition}\label{prop:algH} The analytic and algebraic definitions of the H-invariant agree for special degenerations. \end{proposition}
\begin{proof}
We can apply the formula \eqref{eq:c0int} to the product degeneration $\X_0 \times \C$ induced by the action on $\X_0$, together with \cite[Proposition 4.1]{BWN} to see that
\[ c_0 = \int_X e^{\theta_0}\,\omega_0^n, \]
where $\omega_0$ is any smooth $S^1$-invariant metric on $\X_0$ and $\theta_0$ is the Hamiltonian of the induced vector field. 

What remains is to show that 
\[ b_1 = -\int_{\X_0} \theta_0 e^{h_0}\frac{\omega_0^n}{2(n-1)!}, \] 
but this is essentially standard when the vector field $v$ on $\X$ generates a $\C^*$-action~\cite{GT,RB,WH}, and both sides are linear under scaling $v$.
\end{proof}

\begin{remark} The algebraic definition of the Donaldson-Futaki invariant is $$\DF(\X) =n!\left(b_0 - \frac{2}{n}b_1\right).$$ Applying the finite Jensen's inequality one can prove $\DF(\X) \geq H(\X)$ for special degenerations, which is just the algebraic analogue of Lemma \ref{lem:jensen}, however to characterise the equality case it seems advantageous to use the analytic representations of both quantities. \end{remark}

\section{The H-functional}

Let $X$ be a $\Q$-Fano variety.

\begin{definition} Let $\omega\in c_1(X)$ be a K\"ahler metric, with Ricci potential $h$, and define the \emph{H-functional}  to be $$H(\omega) = \int_X he^h\omega^n.$$
Recall here the normalisation of $h$ so that $e^h$ has average one. It will be useful to extend this definition to the case where $\omega$ is smooth merely on the regular locus $X_{reg}$, with continuous potential globally and also continuous Ricci potential. In this situation, we make the same definition, where $\omega^n$ means the Bedford-Taylor wedge product.
\end{definition}

As we described in the Introduction, this functional goes back to Ding-Tian~\cite{DT92b}, is monotonic along the K\"ahler-Ricci flow~(see Pali~\cite{Pali08}, Phong-Song-Sturm-Weinkove \cite{PSSW3}) and was studied as a functional on the space of K\"ahler metrics in more detail by He~\cite{WH}. 

We now assume $X$ is smooth. By the Jensen inequality we have $H(\omega)\geq 0$, with equality if and only if $\omega$ is a K\"ahler-Einstein metric. More generally if $\omega$ is a K\"ahler-Ricci soliton, with soliton vector field $W$, it is clear that $H(\omega) = H(W)$ as then the Hamiltonian, suitably normalised, equals the Ricci potential.  Moreover, the critical points of $H(\omega)$ are precisely K\"ahler-Ricci solitons, and the gradient flow of the H-functional is simply the K\"ahler-Ricci flow (see \cite[Proposition 2.2]{WH}). The H-functional thus plays the role for the K\"ahler-Einstein problem that the Calabi functional plays for the constant scalar curvature problem. 

The goal of the present section is to prove that the infimum of the H-functional has an algebro-geometric interpretation.

\begin{theorem}\label{equality} We have $$\inf_{\omega \in c_1(X)} H(\omega) =\sup_{\X} H(\X),$$ where the supremum on the right hand side is taken over all $\R$-degenerations (or even just special degenerations). In addition the supremum is achieved by an $\R$-degeneration. \end{theorem}

This will be a consequence of the work of He~\cite{WH} to obtain the inequality $H(\omega) \geq H(\X)$ for any metric $\omega$ and $\R$-degeneration $\X$, as well as the work of Chen-Sun-Wang~\cite{CSW}, to obtain the equality. 

\begin{theorem}\label{He-Theorem} Let $\X$ be an $\R$-degeneration for $X$. Then for any metric $\omega \in c_1(X)$ we have $H(\omega) \geq H(\X)$. 
\end{theorem}
\begin{proof}
First of all, by an approximation argument it is enough to show the result for test-configurations $\X$. 

Next we can assume that $\X$ is normal. The reason is that under normalisation the number $c_0$ does not change (see \cite[Theorem 3.14]{BHJ}), while the number $b_1$ can only decrease (see \cite[Proposition 3.15]{BHJ} or \cite[Proposition 5.1]{RT}). It follows that under normalisation the H-invariant can only increase. 

Let $\varphi_t$ be a geodesic ray in the space of K\"ahler potentials, induced by the normal test-configuration $\X$, with initial metric $\omega$ \cite{PS}. Let us denote the associated (weak) K\"ahler metrics by $\omega_t = \omega + dd^c \varphi_t$ and set 
\[ Y(t) = -\int_X \dot\varphi_t e^{h_{\varphi_t}}\omega_t^n -V\log \left(V^{-1}\int_X e^{-\dot\varphi_t}\omega_t^n\right). \] 

We now show $$H(\omega) \geq \lim_{t\to\infty} Y(t)$$ This inequality is due to He \cite{WH}; for the readers' convenience we give the proof. 

We can normalise the potentials $\varphi_t$ so that
  \begin{equation}\label{eq:intphidot}
 \int_X e^{-\dot\varphi_t}\omega_t^n = V, 
\end{equation}
  since according to Berndtsson~\cite[Proposition 2.2]{Ber09_1} this integral is independent of $t$. 
  Note also that, up to the addition of a constant, $-Y(t)$ is the derivative of the Ding functional along the ray of K\"ahler potentials $\varphi_t$. By Berndtsson~\cite{Ber13}, the Ding functional is convex along geodesics, and so $Y(t)$ is monotonically decreasing in $t$. 

  At the same time, by Jensen's inequality we have
  \[ V^{-1}\int_X (-h - \dot\varphi_0) e^{h}\, \omega^n \leq \log\left( V^{-1} \int_X e^{-h -\dot\varphi_0} e^{h}\omega^n\right) = 0, \]
  and so
  \[ H(\omega) = \int_X h e^{h} \omega^n \geq -\int_X \dot\varphi_0 e^{h}\,\omega^n = Y(0). \]
  By the monotonicity of $Y(t)$ it then follows that
  \begin{equation}\label{eq:Hlower}
 H(\omega) \geq \lim_{t\to\infty} Y(t). 
\end{equation}

What remains is to relate this limit of $Y(t)$ to the H-invariant. For this note first that by Hisamoto~\cite[Theorem 1.1]{TH}, the formula \eqref{eq:c0int}, and \eqref{eq:intphidot}, we have
\[ c_0 = \int_X e^{-\dot\varphi_t}\,\omega_t^n = V, \]
and so from Definition~\ref{defn:algH} we have 
\[ H(\X) = -2 b_1 (n-1)!. \]
Since $-Y(t)$ is, up to addition of a constant, the derivative of the Ding functional along the geodesic ray, the asymptotics of $Y(t)$ can be obtained from Berman~\cite[Theorem 3.11]{RB}. It follows that
\[ H(\X) \leq \lim_{t\to\infty} Y(t), \]
and so with \eqref{eq:Hlower} the proof is complete.  
\end{proof}

The other main ingredient that we use is the following, due to Chen-Wang \cite{CW} and Chen-Sun-Wang \cite{CSW} (see especially \cite[p12, p16]{CSW}). 

\begin{theorem}\label{CSW-Theorem} Let $(X,\omega(t))$ be a solution of the K\"ahler-Ricci flow. The sequential Gromov-Hausdorff limit of $(X,\omega(t))$ as $t\to\infty$ is a $\Q$-Fano variety $Y$, independent of choice of subsequence, which admits a K\"ahler-Ricci soliton with soliton vector field $W_Y$. Assume that $W_Y\ne 0$. Then there exists a ``two-step'' $\R$-degeneration from $X$ to $Y$, i.e. an $\R$-degeneration $\X_{a}$ for $X$ with $\Q$-Fano central fibre $\bar X$, and an $\R$-degeneration $\X_{b}$ for $\bar X$ with central fibre $Y$. The corresponding (real) one-parameter group of automorphisms on $\X_{b}$ is induced by the soliton vector field on $Y$.\end{theorem}

Chen-Sun-Wang prove that the Donaldson-Futaki invariants of $\X_{a}$ and $\X_{b}$ are equal \cite[Proposition 3.5]{CSW}, or equivalently the Futaki invariants of $\bar X$ and $Y$ are equal. We will require an analogous statement for the H-invariant.

\begin{lemma}\label{2stepH} We have $H(\X_{a}) = H(\X_{b})$. \end{lemma}

\begin{proof}
By \cite[Proof of Lemma 3.4, Proposition 3.5]{CSW}, the weight decompositions of $H^0(\bar X, -rK_{\bar X})$ and $H^0(Y,-rK_Y)$ are isomorphic for all sufficiently large and divisible $r$, hence invariants created from these decompositions are equal. The result follows from the algebraic definition of the H-invariant, Definition~\ref{defn:algH}. 
\end{proof}

We now proceed to the proof of Theorem \ref{equality}.

\begin{proof}[Proof of Theorem \ref{equality}] From Theorem~\ref{He-Theorem} we already know that
\[ \inf_{\omega\in c_1(X)} H(\omega) \geq \sup_{\X} H(\X), \]
taking the supremum over all $\R$-degenerations. To complete the proof we show that the $\R$-degeneration $\X_a$ obtained from the K\"ahler-Ricci flow $\omega(t)$ in Theorem~\ref{CSW-Theorem} satisfies
\begin{equation}\label{eq:limHeq} 
 \lim_{t\to\infty} H(\omega(t)) = H(\X_a). 
\end{equation}
To see this, note that $(Y,\omega_Y)$ is the Gromov-Hausdorff limit of $(X,\omega(t))$, and so \begin{equation}\lim_{t \to \infty}H(\omega(t))=H(\omega_Y),\end{equation} where the former quantity is calculated on $X$ and the latter is calculated on $Y$. Here to make sense of $H(\omega_Y)$ we recall how to define the Ricci potential in this situation. The regularity results for K\"ahler-Ricci solitons imply that $\omega_Y$ is smooth on the regular locus $Y_{reg}$, and has continuous potential on $Y$ \cite[Section 3.3]{BWN}. Then as the Ricci potential is uniformly bounded in $C^1$ along the flow~\cite{ST03}, the Ricci potential of $\omega_Y$ on the regular locus $Y_{reg}$ extends to a continuous function on $Y$ which we still call the Ricci potential, and one can define $H(\omega_Y)$ as usual. Moreover, this implies that to prove \eqref{eq:limHeq} one one can work only on the smooth locus. 

Denote by $V$ the soliton vector field on $Y$. What we now show is that $H(V) = H(\omega_Y)$. This would be immediate if $\omega_Y$ were a smooth K\"ahler metric, however the regularity results described above do not imply this. Denote by $h$ the Ricci potential of $\omega_Y$, which is continuous by the above, and pick a smooth K\"ahler metric $\eta\in c_1(Y)$ with Ricci potential $f$ and Hamiltonian $\theta$. It follows from \cite{RB} that $$\int_{Y} h e^{h}\omega_Y^n = \int_Y \theta e^{f}\eta^n,$$ where we have used that $h$ is the Hamiltonian with respect to $\omega_Y$ of $V$. Indeed, both quantities are limit derivatives of a component of the Ding functional \cite{RB,BWN}. For the remaining term, we use that $h$ is continuous and so $$\int_Y e^h\omega_Y^n = \int_Y e^{\theta} \eta^n,$$ as follows from the proof of \cite[Theorem 3.14]{RD}. Thus $H(Y) = H(\X_b)$, and by Lemma \ref{2stepH} $H(\X_a) = H(\X_b)$, so \eqref{eq:limHeq} follows. 

Note that the $\R$-degeneration $\X_a$ has $\Q$-Fano central fiber, and so the approximating test-configurations are actually special degenerations. It follows that $\sup_{\X} H(\X)$ can be computed by considering only special degenerations $\X$. 
\end{proof}

\begin{remark}
We should emphasise that the central fiber $\overline{X}$ of the optimal degeneration does not necessarily agree with the Gromov-Hausdorff limit $Y$ along the K\"ahler-Ricci flow. The first $\mathbb{R}$-degeneration with central fiber $\overline{X}$ is analogous to the Harder-Narasimhan filtration of an unstable vector bundle, while the second $\mathbb{R}$-degeneration with central fiber $Y$ is analogous to the Jordan-H\"older filtration of a semistable bundle. 
\end{remark}

\section{Applications to the K\"ahler-Ricci flow}

We recall the following functionals introduced by Perelman~\cite{Per02}:

\begin{definition} Denote by $S(\omega)$ the scalar curvature of $\omega$. For a smoth function $f$ on $(X,\omega)$ satisfying $$\int_X e^{-f}\omega^n = V,$$ we define the \emph{W-functional} to be $$W(\omega, f) = \int_X (S(\omega) + |\nabla f|^2 + f)e^{-f}\omega^n.$$ The \emph{$\mu$-functional} is defined as $$\mu(\omega) = \inf_{f\in C^{\infty}(X)} W(\omega, f).$$\end{definition} 

The following extends \cite{TZZZ,WH}, who proved a special case of the following result, namely an upper bound over product special degenerations.

\begin{theorem}\label{thm:perelman} We have $$\sup_{\omega\in c_1(X)}\mu(\omega) = nV - \sup_{\X} H(\X).$$
In addition the supremum of $\mu$ is achieved in the limit along the K\"ahler-Ricci flow with any initial metric on $X$. 
\end{theorem}

\begin{proof} It is shown in \cite{WH} that for \emph{all} $\omega \in c_1(X)$ we have $$\mu(\omega)\leq nV - H(\omega),$$ and hence $$\sup_{\omega} \mu(\omega) \leq nV -  \inf _{\omega} H(\omega).$$ Thus by Theorem \ref{equality} we have  
\begin{equation}\label{eq:supmu}
\sup_{\omega} \mu(\omega) \leq nV - \sup_{\X} H(\X).
\end{equation} 

Let $(Y,\omega_Y)$ be the Gromov-Hausdorff limit of $(X,\omega)$ along the K\"ahler-Ricci flow as usual. Let $h_Y$ be the Ricci potential of $\omega_Y$. Then we have have, as in \cite[p17]{CSW} (but with different normalisations),  $$W(\omega_Y,h_Y) = nV - H(Y),$$ hence by Theorem \ref{CSW-Theorem} and Lemma \ref{2stepH} there is a special degeneration $\X_a$ such that $nV - H(\X_a) = W(\omega_Y,h_Y)$. 

What remains is to show that along the K\"ahler-Ricci flow $\omega_t$ we have 
\[ \lim_{t\to\infty} \mu(\omega_t) = W(\omega_Y, h_Y). \]
Let $f_t$ denote the minimiser of the $W$ functional on $(X,\omega_t)$, i.e. $\mu(\omega_t) = W(\omega_t, f_t)$. It is well known~\cite{Rot} that $f_t$ is smooth, and in addition the uniform control of the Sobolev constant~\cite{Ye, Zha} and the scalar curvature~\cite{ST03} along the flow implies that we have a uniform bound $|\Phi_t| < C$, where $\Phi_t = e^{-f_t/2}$. In terms of $\Phi_t$ we have
\begin{equation}\label{eq:WPhi}
W(\omega_t, f_t) = \int_X (S(\omega_t) \Phi_t^2 + 4|\nabla \Phi_t|^2 - 2\Phi_t^2 \ln \Phi_t)\,\omega_t^n. 
\end{equation}

From Bamler~\cite[p. 60]{Bam} and Chen-Wang~\cite[Proposition 6.2]{CW} we know that on any compact subset $K\subset Y_{reg}$ of the regular part of $Y$, we have $f_t \to h_Y$ as $t\to\infty$, and elliptic regularity implies that this convergence holds for derivatives as well. It follows that 
\begin{equation}\label{eq:WK}
\lim_{t\to\infty} \int_K (S(\omega_t) + |\nabla f_t|^2 + f_t)e^{-f}\,\omega_t^n = \int_K (S(\omega_Y) + |\nabla h_Y|^2 + h_Y)e^{-h_Y}\,\omega_Y^n, 
\end{equation}
where we are viewing the $\omega_t$ as defining metrics on $K$ for large $t$, using the smooth convergence of the metrics on the regular part of $Y$. 

We know that $h_Y, \nabla h_Y$ are bounded, and so as we let $K$ exhaust $Y_{reg}$ on the right hand side of \eqref{eq:WK} we recover $W(\omega_Y, h_Y)$. On the left hand side we use the expression \eqref{eq:WPhi} in terms of $\Phi_t$, and the fact that the singular set has codimension at least 4~\cite{Bam}, and so in particular we can choose $K$ with the volume of $Y\setminus K$ being arbitrarily small. Although a priori the gradient $\nabla\Phi_t$ may concentrate on the singular set, it still follows that for any $\epsilon > 0$, we can choose $K$ so that for sufficiently large $t$
\[ \int_K (S(\omega_t) + |\nabla f_t|^2 + f_t)e^{-f}\,\omega_t^n \leq W(\omega_t, f_t) + \epsilon. \]
Exhausting $Y_{reg}$ with compact sets $K$ we obtain
\[ \lim_{t\to\infty} W(\omega_t, f_t) \geq W(\omega_Y, h_Y). \]
Together with \eqref{eq:supmu} this implies $\lim_{t\to\infty} \mu(\omega_t) = W(\omega_Y, h_Y)$, which is what we wanted to show. 
\end{proof}

\begin{corollary} Suppose $X$ is a Fano manifold admitting a K\"ahler-Ricci soliton $\omega_{KRS}$, and let $\omega\in c_1(X)$ be an arbitrary (not necessarily automorphism invariant) K\"ahler metric. The K\"ahler-Ricci flow starting from $\omega$ converges to $\omega_{KRS}$, up to the action of the automorphism group of $X$.
\end{corollary} 

\begin{proof} It is a result of Tian-Zhang-Zhang-Zhu \cite{TZZZ} that if $X$ admits a K\"ahler-Ricci soliton $\omega_{KRS}$ and $\omega(t)$ satisfies the K\"ahler-Ricci flow for arbitrary $\omega(0)$, then provided $$\mu(\omega(t))\to \sup_{\omega \in c_1(X)}\mu(\omega),$$ then the flow converges to $\omega_{KRS}$ modulo the action of automorphisms of $X$. Thus the claim follows from Theorem \ref{thm:perelman}.\end{proof}

\begin{remark}
  In some examples we expect that Theorem~\ref{thm:perelman} can be used to identify the limit of the K\"ahler-Ricci flow on $X$, even if $X$ does not admit a soliton, by finding degenerations that maximize the $\mu$-functional. In their work on the K\"ahler-Ricci flow on $S^2$ with conical singularities, Phong-Song-Sturm-Wang~\cite{PSSW1, PSSW2} use this approach to identify the limiting solitons along the flow. 
\end{remark}

\vspace{4mm}

\end{document}